\documentclass{article}
\usepackage{amsmath,amssymb,amsfonts,amsthm}
\usepackage{bm,graphicx,comment,color}
\usepackage{algorithmic,algorithm}

\usepackage{multirow}
\usepackage{longtable}
\usepackage{multirow}
\usepackage{makecell}
\usepackage{xcolor}
\usepackage{enumerate}

\usepackage{ragged2e}\justifying\let\raggedright\justifying
\usepackage{graphicx,url,pgf}
\usepackage{color}
\usepackage{booktabs}
\usepackage[all]{xy}
\usepackage{latexsym}
\usepackage{graphics}
\usepackage{tabularx}
\usepackage{mathrsfs}
\usepackage{delarray}
\usepackage{subfig}
\usepackage{comment}

\newtheorem{theorem}{Theorem}[section]

\newtheorem{lemma}[theorem]{Lemma}

\newtheorem{example}[theorem]{Example}

\newtheorem{definition}[theorem]{Definition}

\newcommand{\W}{\mathcal W}
\newcommand{\w}{_{_{\mathcal W}}}

\newcommand{\QT}{\mathcal Q\mathcal T}

\newcommand{\M}{\mathcal M}

\newcommand{\N}{\mathcal N}
\newcommand{\one}{{\bf 1}}

\usepackage[numbers,sort&compress]{natbib}

\begin{document}
\title{Algorithms for square root of semi-infinite quasi-Toeplitz $M$-matrices \thanks{The work of the first author was partly supported by  by the National Natural Science Foundation of China under grant No.12001262 and by Jiangxi Provincial Natural Science Foundation under grant No.20224BAB211006. The work of the third author was partly supported by the National Natural Science Foundation of China under grant No.12201591.   Version of \today}} 
\author{
Hongjia Chen\thanks{Department of Mathematics, Nanchang University, China. ({\tt chenhongjia@ncu.edu.cn})} 
\and 
Hyun-Min Kim\thanks{Department of Mathematics, Pusan National University, Korea.({\tt hyunmin@pusan.ac.kr})}
\and 
Jie Meng \thanks{School of Mathemtical Sciences , Ocean University of China, China.({\tt mengjie@ouc.edu.cn})}
}
\date{}\thispagestyle{empty}
\maketitle

\begin{abstract}
A quasi-Toeplitz $M$-matrix $A$ is an infinite $M$-matrix that can be written as the sum of a semi-infinite Toeplitz matrix and a correction  matrix.   This paper is concerned with computing the square root of invertible quasi-Toeplitz $M$-matrices which  preserves the quasi-Toeplitz structure.  We show that the Toeplitz part of the square root can be easily computed   through evaluation/interpolation at the $m$ roots of unity. This advantage allows  to propose  algorithms solely for the computation of correction part, whence we propose a fixed-point iteration and a structure-preserving doubling algorithm. Additionally, we show that the correction part can be approximated by solving a nonlinear matrix equation with coefficients of finite size followed by extending the solution to infinity.   Numerical experiments showing the efficiency of the proposed algorithms are performed..
  \end{abstract}

\section{Introduction}
$M$-matrices in the context of infinite dimensional spaces are called M-operators, which, to our knowledge, were firstly investigated in \cite{IM2}, since then related theoretical properties have been developed  in \cite{AS,MRK,IM2,IM3,IM,PN}. 
Quasi-Toeplitz $M$-matrices are infinite $M$-matrices with an almost Toeplitz structure, they  are encountered in the numerical solution of a quadratic matrix equation \cite{BMM} involved in  2-dimensional Quasi-Birth-Death (QBD) stochastic processes \cite{motyer-taylor}  and  are recently studied in \cite{qtm_laa} in terms of their theoretical and computational properties.

In this paper, we are interested in the quasi-Toeplitz $M$-matrices that belongs to the class  $\mathcal{QT}_{\infty}=\{T(a)+E: a(z)\in \W,  E\in \mathcal{K}_d(\ell^{\infty}) \}$, where $T(a)$ is a semi-infinite Toeplitz matrix associated with the function $a(z)=\sum_{i\in \mathbb Z}a_iz^i$ in the sense that $(T(a))_{i,j}=a_{j-i}$, $\W$ is the Wiener algebra, defined as the set $\mathcal W=\{a(z)=\sum_{i\in \mathbb Z}a_iz^i:z\in \mathbb T,  \|a\|_{\w}:=\sum_{i\in \mathbb Z}|a_i|<\infty\}$,  and $\mathcal K_d(\ell^{\infty})=\{E=(e_{i,j})_{i,j\in\mathbb Z^+}:\lim_i\sum_{j=1}^{\infty}|e_{i,j}|=0\}$.  It has been proved in \cite[Theorem 2.16]{BMML2020} that the  class $\mathcal{QT}_{\infty}$  is a Banach algebra with the infinity matrix norm $\|\cdot\|_{\infty}$, which turns out to be  $\| A\|_{\infty } = \sup_i\sum_{j=1}^{\infty}| a_{i,j}|$ for $A=(a_{i,j})_{i,j\in \mathbb Z^+}$.  For  $A=T(a)+E\in \QT_{\infty}$,   $T(a)$ is called the Toeplitz part with a symbol $a$,   $E$ is called the correction part. Matrices in the class $\QT_{\infty}$ have rich and elegant theoretical and computational properties,  we refer the reader to \cite{arXiv,QTeig,QTeig2,BSM1,BSM2,BMML,BMML2020,BSR,BMM,kim_meng,leonardo} for more details.

  For a quasi-Toeplitz $M$-matrix $A=T(a)+E_A\in \QT_{\infty}$,  it has been proved in \cite{qtm_laa} that if $A$ is an (invertible) $M$-matrix, then $T(a)$ is also an (invertible) $M$-matrix. Moreover, it shows that if  $A$ is  invertible,   there exists a unique quasi-Toeplitz $M$-matrix $S=T(s)+E_S\in \QT_{\infty}$ such that  $A=S^2$. Concerning the computation of  matrix $S$,  Binomial iteration and Cyclic Reduction (CR) algorithm have been proposed in \cite{qtm_laa}, where the CR algorithm  seems to be better suited in the numerical computations. However, both the Binomial iteration  and the CR algorithm exploit the quasi-Toeplitz structure indirectly by performing approximate operations of semi-infinite quasi-Toeplitz matrices in the format. It would be natural to ask  whether the quasi-Toeplitz structure can be fully exploited to propose more efficient algorithms.  
  
  Suppose $B=T(b)+E_B\in \QT_{\infty}$ satisfies $(I-B)^2=A$, where $A=T(a)+E_A$ is  a given quasi-Toeplitz $M$-matrix, then we have for the symbols of the Toeplitz parts that $(1-b(z))^2=a(z)$. Observe that for a positive integer $n>0$, there is always a unique Laurent polynomial $\hat{b}(z)=\sum_{i=-n+1}^n\hat{b}z^i$ that interpolates $b(z)$ at  the $2n$ roots of unity.  Based on the technic of evaluation/interpolation, we investigate computation of the coefficients $b_i$ of $b(z)=\sum_{i\in \mathbb Z}b_iz^i$, so that the Toeplitz part $T(b)$ of the quasi-Toeplitz $M$-matrix $A$ can be easily obtained. 
  
 Concerning the computation of the correction part,  we propose  a fixed-point iteration with a  linear convergence rate, and a structure-preserving doubling  algorithm,  which is of quadratic convergence rate. Moreover, we show that the correction part can be approximated by extending a finite size matrix to infinity, where the finite size matrix solves a nonlinear matrix equation. 
 Numerical experiments show that the proposed algorithms provide convergence acceleration in terms of CPU times comparing with the Binomial iteration  and CR algorithm proposed in  \cite{qtm_laa}, both of which keep the whole quasi-Toeplitz matrices in the computations.   

This paper is organized as follows. In the remaining part of this introduction, we recall some definitions and properties concerning quasi-Toeplitz matrices and  $M$-matrices.  Sections 2 and 3 concern with algorithms that fully exploit the quasi-Toeplitz  structure of square root of invertible quasi-Toeplitz $M$-matrices, in Section 2 we show how the Toeplitz part is computed, while in Section 3, we design and analyze the convergence of algorithms that are applicable in computing the correction part. In Section 4, we show that the correction part can be approximated by extending to infinity of the solution of a nonlinear matrix equation with finite size coefficients. In Section 5, we show by numerical examples the efficiency of the proposed algorithms.

  \subsection{Preliminary concepts}
 Let  $\ell^{\infty}$ be the space of sequences $\{x=(x_1,x_2,\ldots)\}$ such that  $\sup_{i\in \mathbb Z^+}|x_i|<\infty$,  
one can see that quasi-Toeplitz $M$-matrices in the class $\QT_{\infty}$ are bounded linear operators  from $\ell^{\infty}$ to $\ell^{\infty}$. Denote by  $\mathcal B(\ell^{\infty})$ the Banach space of bounded linear operators from $\ell^{\infty}$ to itself,  we first recall definition of $M$-operators on $\mathcal B(\ell^{\infty})$.  For  definition of more general $M$-operators on a real partially ordered Banach space, we refer the reader to \cite{MRK,IM,PN} and the references therein.  $M$-operators on the Banach space $\ell^{\infty}$ are defined as 
  \begin{definition}
  	 An operator $A\in \mathcal B(\ell^{\infty})$ is said to be a $Z$-operator if $A=sI-P$, with $s\geq 0$, $P(\ell^{\infty}_+)\subseteq \ell^{\infty}_+$, where $\ell^{\infty}_+=\{x=(x_i)_{i\in \mathbb Z^+}\in \ell^{\infty}: x_i\geq 0 \ for\ all\ i\}$. A  $Z$-operator is said to be an $M$-operator if $s\geq \rho(P)$, where $\rho(P)$ is the spectral radius of $P$. $A$ is an invertible $M$-operator if $s>\rho(P)$.  
  \end{definition}

As matrices in $\QT_{\infty}$ can be represented  as a matrix of infinite size, we keep using the term $M$-matrix when referring $M$-operators in  $\QT_{\infty}$. This way, a matrix $A\in \QT_{\infty}$ is said to be an $M$-matrix if $A=\beta I-B$ with $B\geq 0$ and $\beta\geq \rho(B)$, and $A$ is invertible if $\beta>\rho(B)$. Here $B\geq 0$ means that $B$ is an elementwise nonnegative infinite matrix. 

\medskip

The following lemma contains a collection of properties of quasi-Toeplitz matrices and  quasi-Toeplitz $M$-matrices, where properties (i) and (ii) have been proved in \cite{AB}, while  properties (iii) - (v) can be found from \cite{qtm_laa}. 
\begin{lemma}\label{summary}
If $A=T(a)+E_A\in \QT_{\infty}$ and $B = T(b) + E_B \in \QT_{\infty}$, then the following properties hold:

\begin{itemize}
\item[{\rm i)}] $AB = T(ab) - H(a^-)H(a^+)\in  \QT_{\infty}$, where  $(H(a^-))_{i,j}=(a_{-i-j+1})_{i,j\in \mathbb Z^+}$ and $(H(a^+))_{i,j}=(a_{i+j-1})_{i,j\in \mathbb Z^+}$;
\item[{\rm ii)}]  it holds that  $\|a\|_{\w}=\|T(a)\|_{\infty}\leq \|A\|_{\infty}$;
\item[{\rm iii)}]   $T(a)\geq 0$ if $A\geq 0$.
\item[{\rm iv)}]  $\|a\|_{\w}=a(1)$ if $T(a)\geq 0$.
\item[{\rm v)}]   T(a) is an (invertible) $M$-matrix if $A$ is an (invertible) $M$-matrix.
\end{itemize}
\end{lemma}

The following lemma shows that an invertible $M$-matrix in the class $\QT_{\infty}$ admits a unique  quasi-Toeplitz $M$-matrix  as a square root. 
\begin{lemma}\cite[Theorem 3.6]{qtm_laa}\label{thm:uniquesquareroot}
Suppose $A=\beta(I-A_1)\in\QT_{\infty}$ satisfies $\beta>0$, $A_1\geq 0$ and $\|A_1\|_{\infty}<1$, then there is  a unique $B\in \QT_{\infty}$ such that $B\geq 0, \|B\|_{\infty}<1$, and $(I-B)^2=I-A_1$.
\end{lemma}

For quasi-Toeplitz $M$-matrix $A=\gamma(I-A_1)\in \QT_{\infty}$ such that $A_1\geq 0$ and $\|A_1\|_{\infty}<1$, it can be seen from Lemma \ref{thm:uniquesquareroot} that  it suffices to compute matrix $B$ such that $(I-B)^2=I-A_1$. In what follows,  we propose algorithms for computing the Toeplitz part and the correction part of matrix $B$. 

\medskip

\section{Computing the Toeplitz part}\label{sec:symbol}
 Observe that the Toeplitz part $T(b)$ is uniquely determinate by the coefficients $b_j$ of the symbol $b(z)=\sum_{j\in \mathbb Z}b_jz^j$. In this section, we show that $b(z)$ can be approximated by $\hat{b}(z)=\sum_{i=-n+1}^n\hat{b}_jz^j$ in the sense that $\|b-\hat{b}\|_{\w}\leq c\epsilon$ for some constant $c$ and a given tolerance $\epsilon$.

Suppose $B=T(b)+E_B$ satisfies  $\gamma(I-B)^2=A$, where $A=\gamma(I-A_1)\in \QT_{\infty}$ is such that $A_1\geq 0$ and $\|A_1\|_{\infty}<1$. Suppose $T(a)$  is the Toeplitz part of $A$,   we have from  property (i) of Lemma \ref{summary} that  $\gamma(1-b(z))^2=a(z)$,
that is, 
\begin{align}\label{eq:symbol}
a(z)/\gamma=b(z)^2-2b(z)+1,
\end{align}
from which we obtain $b(z)=1\pm \sqrt{a(z)/\gamma}$.  Since  $A_1\geq 0$, in view of properties (ii)-(iv) of Lemma \ref{summary}, we have $a_1(1)=\|a_1\|_{\w}\leq \|A_1\|_{\infty}<1$, where $a_1(z)$ is the symbol of the Toeplitz part of $A_1$,  hence we deduce that  $a(1)=\gamma(1-a_1(1))>0$.  On the other hand, it follows from $B\geq 0$ that  $b(1)=\|b\|_{\w}= \|T(b)\|_{\infty}\leq \|B\|_{\infty}<1$, which, together with  $\sqrt{a(1)/\gamma}>0$, implies that  $b(1)=1-\sqrt{a(1)/\gamma}$ and therefore  $b(z)=1-\sqrt{a(z)/\gamma}$.
\smallskip

Let $n>0$ be a positive integer, set $m=2n$, then there is always a unique Laurent series $\hat{b}(z)=\sum_{j=-n+1}^n\hat{b}_jz^j$ such that $\hat{b}(\omega_m^{\ell})=b(\omega_m^{\ell})$, $\ell=-n+1,\ldots,n$, where $\omega_m$ is
the principal $m$-th root of 1, that is, $\omega_m=\cos\frac{2\pi}{m}+{\bf i}\sin\frac{2\pi}{m}$.   Based on the evaluation/interpolation technique, where the interpolation can be done by the means of the Fast Fourier Transform (FFT),   an approximation $\hat{b}_i$, $i = -n + 1, . . . , n$, to the coefficients $b_i$ of $b(z)$ can be obtained. Since $B\geq 0$, we have from property (iii) of Lemma \ref{summary} that $T(b)\geq 0$, so that  $b(z)=\sum_{i\in \mathbb Z}b_iz^i$ has nonnegative coefficients. If in addition $b''(z)\in \W$, 
the following lemma provides a bound to $|\hat{b}_i-b_i|$.

\begin{lemma}\cite[Lemma 3.1]{BMM}\label{lem:b}
For $g=\sum_{i\in \mathbb Z}g_iz^i\in \mathcal W$ with nonnegative coefficients, let $\hat{g}(z)=\sum_{j=-n+1}^n \hat{g}_jz^j$ be the Laurent polynomial interpolating $g(z)$ at the $m$-th roots of 1, i.e., $g(w_m^i)=\hat{g}(w_m^i)$ for $i=-n+1,\ldots, n$, where $m=2n$. If $g''(t)\in \mathcal W$, then $g''(1)\geq 0$ and 
\begin{align*}
g''(1)-\hat{g}''(1)\geq 2n\big(\sum_{j<-n+1}g_j+\sum_{j>n}g_j\big).
\end{align*}
Moreover, $0\leq \hat{g}_j-g_j\leq \frac{1}{2n}(g''(1)-\hat{g}''(1))$ for $j=-n+1,\ldots,n$.
\end{lemma}

For  $\hat{b}(z)=\sum_{j=-n+1}^n \hat{b}_jz^j$ interpolating $b(z)$ at $\omega_m^i$ for $i=-n+1,\ldots,n$, suppose $b''(z)\in \mathcal W$ and $b''(1)>0$, we have from Lemma \ref{lem:b} that 
\begin{equation}\label{tail0}
b''(1)-\hat{b}''(1)\geq 2n\big(\sum_{j<-n+1}b_j+\sum_{j>n}b_j\big),
\end{equation}
and 
\begin{equation}\label{tail}
|\hat{b}_j-b_j|\leq \frac{1}{2n}(b''(1)-\hat{b}''(1)), \ j=-n+1,\ldots,n. 
\end{equation}
If $b''(1)-\hat{b}''(1)<\epsilon$ for a given tolerance $\epsilon  > 0$,   we have from \eqref{tail} that $|b_j-\hat{b}_j|\leq \epsilon/(2n)$ for $j=-n+1,\ldots,n$, which together with \eqref{tail0} implies that 
\begin{align*}
\|b-\hat{b}\|_{\w}&=\sum_{j=-n+1}^n|b_j-\hat{b}_j|+\sum_{j<-n+1}b_j+\sum_{j>n}b_j\\
&\leq \epsilon+\frac{1}{2n}(b''(1)-\hat{b}''(1))\\
&\leq (1+\frac{1}{2n})\epsilon.
\end{align*}
Hence, in the computation of  $\hat{b}_j, j=-n+1,\ldots, n$, under the evaluation/interpolation scheme,  the approximation is accurate enough if $b''(1)-\hat{b}''(1)<\epsilon$. Actually, the values of $b''(1)-\hat{b}''(1)$ can be easily obtained. Indeed,  once the coefficients $\hat{b}_j$ of $\hat{b}(z)=\sum_{j=-n+1}^n\hat{b}_jz^j$ are computed, one can easily obtain $\hat{b}''(1)=\sum_{j=-n+1}^nj(j-1)\hat{b}_j$. On the other hand,    we have from equation \eqref{eq:symbol} that
\begin{equation*}\label{twice}
b'(z)=\frac{a'(z)}{2\gamma(b(z)-1)}\  {\rm and} \ b''(z)=\frac{a''(z)-2\gamma(b'(z))^2}{2\gamma(b(z)-1)},
\end{equation*}
from which we easily obtain $b'(1)$ and $b''(1)$. 
\smallskip

Observe that equation \eqref{eq:symbol} is a special case of the quadratic  equation $$a_1(z)g(z)^2+(a_0(z)-1)g(z)+a_{-1}(z)=0,$$ where $a_i(z)$ for $i=-1,0,1$ are known functions in the class $\W$ and $g(z)$ is the function to be determined. Algorithms for computing the approximations of the coefficients of $g(z)$ has been proposed in \cite{BMM}, based on which we propose the following Algorithm \ref{alg:b} that is more efficient in computing the coefficients  $\hat{b}_j$  of the Laurent series $\hat{b}(z)=\sum_{j=-n+1}^n\hat{b}_jz^j$,  so that we get an approximation $T(\hat{b})$ to the Toeplitz part $T(b)$ in the sense that $\|T(b)-T(\hat{b})\|_{\infty}=\|b-\hat{b}\|_{\w}\leq (1+\frac{1}{2n})\epsilon$ for a given tolerance $\epsilon$.

\begin{algorithm}  %[H]
\caption{Approximation of $b(z)$}
\label{alg:b}
%\DontPrintSemicolon
 \begin{algorithmic}[1]
 \REQUIRE{The coefficients of $a(z)$, a scalar $\gamma$ such that $A=\gamma(I-A_1)$  and a tolerance $\epsilon>0$.}
\ENSURE{Approximations $\hat b_j$, $j=-n+1,\ldots,n$,   to the coefficients $b_j$ of $b(z)$
 such that $|\hat{b}_j-b_j|\leq \epsilon/(2n)$.}
 
 \STATE{Set n=4, and compute $b(1)=1-\sqrt{a(1)/\gamma}$ and $b'(1)=\frac{a'(1)}{2\gamma(b(1)-1)}$ and $b''(1)=\frac{a''(1)-2\gamma(b'(1))^2}{2\gamma(b(1)-1)}$; 
 }

\STATE{Set $m=2n$ and $w_m=\cos\frac{2\pi}{m}+{\bf i} \sin\frac{2\pi}{m}$. Evaluate $a(z)$ at $z=w_m^{i}$ for $i=-n+1,\ldots,n$;
}
\STATE{For $i=-n+1,\ldots, n$, compute  $s_i=1-\sqrt{a(\omega_m^{i})/\gamma}$;
}
\STATE{Interpolate the values $s_i$, $i=-n+1,\ldots, n$, by means of FFT and obtain the coefficients $\hat{b}_j$ of the Laurent polynomial $\hat{b}(z)=\sum_{j=-n+1}^n\hat{b}_jz^j$ such that $b(w_m^{i})=\hat{b}(w_m^{i})$, $i=-n+1,\ldots,n$;
}
\STATE{Compute $\hat{b}''(1)=\sum_{j=-n+1}^nj(j-1)\hat{b}_j$ and $\delta_m=b''(1)-\hat{b}''(1)$;
}
\STATE{If $\delta_m< \epsilon$ then exit, else set $n=2n$ and compute from Step 2.
}
\end{algorithmic}
\end{algorithm}

It can be seen that the overall computational cost  of Algorithm \ref{alg:b} is $O(n \log n)$ arithmetic operations. Now the Toeplitz part of matrix $B$  is approximated by $T(\hat{b})$,  it remains to compute the correction part of $B$ in order to complete the computation of the square root. We show this subject in next section.

\section{Computing the correction part}

Suppose $A=\beta(I-A_1)\in \QT_{\infty}$, where $A_1\geq 0$ and $\|A_1\|_{\infty}<1$, then for $B=T(b)+E_B\geq 0$ and $\|B\|_{\infty}<1$ such that $(I-B)^2=I-A_1$, we design and analyze the convergence of a fixed-point iteration and a structure-preserving doubling algorithm  that can be used for the computation of $E_B$.

\subsection{Fixed-point iteration}

 Consider the nonlinear matrix equation 
 \begin{equation*}
(I-T(b)-X)^2=I-A_1
 \end{equation*}  
which  can be equivalently written as 
\begin{equation}\label{eq:sr}
	X^2-(I-T(b))X-X(I-T(b))+Q=0,
\end{equation}
where $Q=A_1+T(b)^2-2T(b)$.  It is clear that $E_B$ solves equation \eqref{eq:sr}.  On the other hand, it follows from Lemma  \ref{thm:uniquesquareroot} that $I-A_1$ allows a unique  quasi-Toeplitz $M$-matrix as a square root, so that $E_B$ is the unique solution of equation \eqref{eq:sr} such that    $T(b)+E_B\geq 0$ and $\|T(b)+E_B\|_{\infty}<1$.  \medskip

Observe that  equation \eqref{eq:sr} can be equivalently  written as $X=(2I-T(b)-X)^{-1}(Q+XT(b))$, from which  we propose  the following iteration  
\begin{align}\label{ifpi}
X_{k+1}=(2I-T(b)-X_k)^{-1}(Q+X_kT(b))
\end{align}
with $X_0=0$. We show that the sequence $\{X_k\}$ converges to $E_B$. To this end, we first show the following result. 

\begin{theorem}\label{thm:fixed_se}
Let $A=\beta(I-A_1)\in \mathcal{QT}_{\infty}$ with $A_1\geq 0$ and $\|A_1\|_{\infty}<1$. Suppose $B=T(b)+E_B\in \QT_{\infty}$ is the unique quasi-Toeplitz matrix such that $B\geq 0$, $\|B\|_{\infty}<1$, and $(I-B)^2=I-A_1$. Then, the sequence $\{X_k\}$  generated by  iteration \eqref{ifpi} satisfies
\begin{itemize}
\item[{\rm (i)}] the sequence $\{X_k\}$ is well defined;
%\item[{\rm (ii)}]  $X_k=F_k(A_1)-T(b)$, where $F_0(A_1)=T(b)$ and $F_{k+1}(A_1)=(2I-F_{k}(A_1))^{-1}A_1$;
\item[{\rm (ii)}] $T(b)+X_k\geq 0$ and $\|T(b)+X_k\|_{\infty}<1$. 
\end{itemize}
\end{theorem}

\begin{proof}

Concerning item (i), observe that $X_{k+1}$ is well defined as long as $2I-T(b)-X_k$ is invertible. It follows from \cite[Lemma 3.1.5]{RV} that  $2I-T(b)-X_k$ is invertible if $\|T(b)+X_k\|_{\infty}<2$,  which can be verified if item (ii) is true.  Hence, it suffices to prove item (ii).
		
We prove  item (ii) by induction. For $k=0$, we have $T(b)+X_0=T(b)\geq 0$,   where the inequality follows from property (iii) of Lemma \ref{summary} and the fact $B\geq 0$. On the other hand, we have from property (ii) of Lemma \ref{summary} that $\|T(b)+X_0\|_{\infty}\leq \|B\|_{\infty}<1$. For the inductive step, assume that $T(b)+X_k\geq 0$ and $\|T(b)+X_k\|_{\infty}<1$, we show that  $T(b)+X_{k+1}\geq 0$ and $\|T(b)+X_{k+1}\|_{\infty}<1$. 
	
	 Observe that
		\begin{align*}
		X_{k+1}&=(2I-T(b)-X_k)^{-1}(A_1-(2I-T(b)-X_k)T(b))\\
		&=(2I-T(b)-X_k)^{-1}A_1-T(b), 
	\end{align*}
	from which we have 
	$$T(b)+X_{k+1}=(2I-T(b)-X_k)^{-1}A_1. $$
	
	It follows from \cite[Lemma 3.1.5]{RV} that  
	\begin{equation*}\label{inver}
		(2I-T(b)-X_k)^{-1}=\frac{1}{2}\sum_{i=0}^{\infty}\big(\frac12(T(b)+X_k)\big)^i,
	\end{equation*}
	so that   $(2I-T(b)-X_k)^{-1}\geq 0$ since $T(b)+X_k\geq 0$. Recall that $A_1\geq 0$,  we thus have $(2I-T(b)-X_k)^{-1}A_1\geq 0$, that is, $T(b)+X_{k+1}\geq 0$.
	\smallskip
	
It remains to show $\|T(b)+X_{k+1}\|_{\infty}<1$. Observe that
\[
\begin{aligned}
\|T(b)+X_{k+1}\|_{\infty}&=\|(2I-T(b)-X_k)^{-1}A_1\|_{\infty}\\
&\leq \|(2I-T(b)-X_k)^{-1}\|_{\infty}\|A_1\|_{\infty}\\
&\leq \frac{\|A_1\|_{\infty}}{2-\|T(b)+X_k\|_{\infty}},
\end{aligned}
\]
where the last inequality holds since 
\begin{align}\label{inver2}
\|(2I-T(b)-X_k)^{-1}\|_{\infty}&\leq \frac12\sum_{i=0}^{\infty}\big(\frac12\|T(b)+X_k\|_{\infty}\big)^i\notag \\
&=\frac{1}{2-\|T(b)+X_k\|_{\infty}}.
\end{align}
Recall that $\|T(b)+X_k\|_{\infty}<1$ and $\|A_1\|_{\infty}<1$,  one can check that
$$\frac{\|A_1\|_{\infty}}{2-\|T(b)+X_k\|_{\infty}}<1,$$ that is,  $\|T(b)+X_{k+1}\|_{\infty} <1$. 
\end{proof}

The following result shows the convergence of  sequence $\{X_k\}$. 
\begin{theorem}
Let $A=\beta(I-A_1)\in \mathcal{QT}_{\infty}$ with $A_1\geq 0$ and $\|A_1\|_{\infty}<1$. Suppose $B=T(b)+E_B\in \QT_{\infty}$ is the unique quasi-Toeplitz matrix such that $B\geq 0$, $\|B\|_{\infty}<1$ and $(I-B)^2=I-A_1$.	Then the sequence $\{X_k\}$ generated by  iteration \eqref{ifpi}  converges to $E_B$ in the sense that $\lim_{k\rightarrow \infty}\|E_B-X_{k}\|_{\infty}=0$.
\end{theorem}
 
\begin{proof}
Let $E_k=E_B-X_k$, a direct computation yields
\begin{equation*}%\label{eq:ek}
E_{k+1}=(2I-T(b)-X_k)^{-1}E_kB,
\end{equation*}
which, together with \eqref{inver2}, yields
\begin{equation}\label{in:Ek}
\|E_{k+1}\|_{\infty}\leq \frac{\|B\|_{\infty}}{2-\|T(b)+X_k\|_{\infty}}\|E_k\|_{\infty}.
\end{equation}
Since $\|T(b)+X_k\|_{\infty}<1$,  it follows that $\frac{\|B\|_{\infty}}{2-\|T(b)+X_k\|_{\infty}}<\|B\|_{\infty}$, so that 
$$\|E_{k+1}\|_{\infty}\leq \|B\|_{\infty}\|E_k\|_{\infty}\leq \|B\|_{\infty}^k\|E_0\|_{\infty}.$$ Since  $\|B\|_{\infty}<1$, it implies that $\lim_{k\rightarrow \infty}\|E_B-X_k\|_{\infty}=0$.
\end{proof}

We may observe from inequality \eqref{in:Ek} that  the sequence $\{X_k\}$ generated by iteration \eqref{ifpi} satisfies $\|X_{k+1}-E_B\|_{\infty}\leq \frac{\|B\|_{\infty}}{2-\|T(b)+X_k\|_{\infty}}\|X_k-E_B\|_{\infty}$. The fact  
$\frac{\|B\|_{\infty}}{2-\|T(b)+X_k\|_{\infty}}<\|B\|_{\infty}$ may provide some insights to say that the  fixed-point iteration \eqref{ifpi}, which is used for the computation of the correction part, converges faster than the Binomial iteration \cite{qtm_laa} in the computation of the whole square root, as the sequence $\{Y_k\}$ generated by the Binomial iteration $Y_{k+1}=\frac{1}{2}(A_1+Y_k^2)$ with $Y_0=0$ satisfies that $\|Y_{k+1}-B\|_{\infty}\leq \|B\|_{\infty}\|Y_k-B\|_{\infty}$.

%%%%%
\subsection{Structure-preserving Doubling Algorithm} 
We show that a structure-preserving doubling algorithm (SDA)  is applicable in the computation of $E_B$ such that $(I-T(b)-E_B)^2=A$, where $A$ is an invertible quasi-Toeplitz $M$-matrix. This method has been motivated by the ideas in \cite{dario_beatrice}, where the SDA that enables refining an initial approximation is applied to solve quadratic matrix equations with quasi-Toeplitz coefficients. We fist recall  the design and  convergence analysis of SDA. For more details of SDA, we refer the reader to \cite{dario_beatrice}, \cite[Chapter 5]{DBB} and \cite{Rencang-Li}.

In the finite dimensional space, the design of SDA is based on a linear pencil $M-\lambda N$, where $M$ and $N$ are $2n \times 2n$ matrices  of the  form 
\begin{equation}\label{eq:ssf1}
M=\left[\begin{array}{cc}E& O\\ -P & I \end{array}\right], \quad N=\left[\begin{array}{cc} I& -Q\\ O& F\end{array}\right],
\end{equation}
where $E, F, P ,Q$ are $n\times n$ matrices, $I$  and $O$ are, respectively, the $n\times n$ identity matrix and the zero matrix. Suppose there are $n\times n$ matrices $X$ and $W$ such that 
\begin{equation*}
M\left[\begin{array}{c}I\\ X \end{array}\right]=N\left[\begin{array}{c}I\\ X \end{array}\right]W. 
\end{equation*}
Then, the SDA consists in computing the sequences defined as
\begin{equation}\label{eq:SDA} 
\begin{aligned}
E_{k+1}&=E_k(I-Q_kP_k)^{-1}E_k\\
P_{k+1}&=P_k+F_k(I-P_kQ_k)^{-1}P_kE_k;\\
F_{k+1}&=F_k(I-P_kQ_k)^{-1}F_k;\\
Q_{k+1}&=Q_k+E_k(I-Q_kP_k)^{-1}Q_kF_k,
\end{aligned}
\end{equation}
where $E_0=E, F_0=F, P_0=P$ and $Q_0=Q$. 
\medskip

We mention that the scheme \eqref{eq:SDA} is quite related to the forms of matrices $M$ and $N$ in \eqref{eq:ssf1}, which is called the standard structured form-I. For different forms, say the standard structured form-II (see \cite[Chapter 5]{DBB}), different schemes can be obtained.   

Concerning the convergence results of SDA, it has been proved in \cite{dario_beatrice} that 
\begin{lemma}\cite[Theorem 2]{dario_beatrice}\label{lem:conver_sda}
Let $X,Y, W, V$ be $n\times n$ matrices such that 
\begin{equation*}
M\left[\begin{array}{c}I\\ X \end{array}\right]=N\left[\begin{array}{c}I\\ X \end{array}\right]W, \quad M\left[\begin{array}{c}Y\\ I \end{array}\right]V=N\left[\begin{array}{c}Y\\ I \end{array}\right],
\end{equation*}
and  it satisfies that $\rho(W)\leq 1$, $\rho(V)\leq 1$,  $\rho(W)\rho(V)<1$. If the scheme \eqref{eq:SDA} can be carried out with no breakdown, then $\lim_k\|X-P_k\|^{1/{2^k}}\leq \rho(W)\rho(V)$ and $\lim_k\|Y-Q_k\|^{1/{2^k}}\leq \rho(W)\rho(V)$.
\end{lemma}

Concerning the feasibility of SDA in the infinite dimensional spaces, it has been shown in \cite[page 11]{dario_beatrice} that  the  convergence results of SDA still hold when matrices belonging to the Banach algebra  $\QT_{\infty}$. We are ready to show how  SDA can be applied in the computation of  $E_B$.

\medskip

Suppose $A=I-A_1\in \QT_{\infty}$ is such that $A_1\geq 0$ and $\|A_1\|_{\infty}<1$, we have from Lemma \ref{thm:uniquesquareroot} that the matrix equation 
\begin{align}\label{eq:square}
(I-X)^2=I-A_1
\end{align}
 has a unique nonnegative solution $B\in \QT_{\infty}$ satisfying $\|B\|_{\infty}<1$.   Observe that equation \eqref{eq:square} can be equivalently written as \begin{align}\label{eq:quadratic}
X^2-2X+A_1=0,
\end{align}
so that $B$ solves equation \eqref{eq:quadratic} and is the unique solution such that $B\geq 0$ and $\|B\|_{\infty}<1$. 
Let $V=(2I-B)^{-1}$, it is easy to check that $V$ solves the quadratic matrix equation 
\begin{align}\label{eq:quadratic2}
A_1Y^2-2Y+I=0.
\end{align}
Moreover,  we have  $V=\frac12\sum_{i=0}^{\infty}(\frac12B)^i\geq 0$ and $\|V\|_{\infty}\leq \frac{1}{2}\sum_{i=1}^{\infty}(\frac12\|B\|_{\infty})^i=\frac{1}{2-\|B\|_{\infty}}<1$. 

\smallskip
Suppose $T(b)$ with $b\in \mathcal W$ is the Toeplitz part of $B$, replacing $X$ by $T(b)+H$ in equation \eqref{eq:quadratic} results in the following quadratic matrix equation 
\begin{equation}\label{eq:quadratic3}
H^2+(T(b)-2I)H+HT(b)+R=0,
\end{equation}
where $R=T(b)^2-2T(b)+A_1$. Then, equation \eqref{eq:quadratic3} can be equivalently written  as 
\[
\widetilde{M}\left[\begin{array}{cc} I\\ H  \end{array}\right]=\widetilde{N}\left[\begin{array}{cc} I\\ H  \end{array}\right]B,
\]
where $\widetilde{M}=\left[\begin{array}{cc} T(b)& I\\ -R & 2I-T(b) \end{array}\right]$ and $\widetilde{N}=\left[\begin{array}{cc} I&0\\ 0&I \end{array}\right]$. 
\medskip

%On the other hand, it can be verified that 
%\[
%\widetilde{M}\left[\begin{array}{c}Y \\ I \end{array}\right]Z=\widetilde{N}\left[\begin{array}{c}Y \\ I \end{array}\right], 
%\]
%where $Y=V(I-T(b)V)^{-1}, Z=(I-T(b)V)V(I-T(b)V)^{-1}$. 

According to \cite[Theorem 3]{dario_beatrice}, the pencil $\widetilde{M}-\lambda \widetilde{N}$ can be transformed into the pencil $\M-\lambda \N$, where $\M$ and $\N$ are of the form 
\begin{equation*}
\M=\left[\begin{array}{cc}SA_1& 0\\-SR & I\end{array}\right],\quad  \N=\left[\begin{array}{cc}I &-S\\ 0 &S\end{array} \right],
\end{equation*}
where $S=2I-T(b)$. It can be seen that $\M$ and $\N$ are of the same forms as those in \eqref{eq:ssf1}, 
and  we have 
$$
\M\left[\begin{array}{c}I \\H \end{array}\right]=\N \left[\begin{array}{c}I \\H \end{array}\right]B,$$
so that  SDA can be applied to the pencil $\M-\lambda \N$, which  consists of 
computing the sequences as defined in the scheme \eqref{eq:SDA} by setting 
\[
P_0=E_0=(2I-T(b))^{-1}B,\quad  \quad Q_0=F_0=S.
\]

On the other hand, it can be verified that  the matrices $\M$ and $\N$ also satisfy 
\begin{equation}\label{eq:H}
\M \left[\begin{array}{c}Y \\I  \end{array}\right]Z=\N \left[\begin{array}{c}Y \\I \end{array}\right], 
\end{equation}
where $Y=V(I-T(b)V)^{-1}, Z=(I-T(b)V)V(I-T(b)V)^{-1}$. It can be seen that $Z$ has the same spectrum as $V$ so that  $\rho(Z)=\rho(V)\leq \|V\|_{\infty}<1$,  we then have from the fact $\rho(B)\leq \|B\|_{\infty}<1$ that $\rho(B)\rho(Z)<1$. Hence, according to Lemma \ref{lem:conver_sda}, we obtain the following convergence result of SDA when applying to the pencil $\M-\lambda \N$. 
\begin{theorem}
For $A=I-A_1\in \QT_{\infty}$ such that $A_1\geq 0$ and $\|A_1\|_{\infty}<1$, suppose $I-B$ with $B=T(b)+E_B\in \QT_{\infty}$ is the unique quasi-Toeplitz $M$-matrix such that $(I-B)^2=A$. If the  scheme \eqref{eq:SDA} can be carried out with no breakdown, then the sequence $\{P_k\}$ converges to $E_B$ and it satisfies $\lim_k\|E_B-P_k\|^{1/{2^k}}\leq \rho(B)\rho(Z)$, where $Z=(I-T(b)V)(2I-B)^{-1}(I-T(b)V)^{-1}$ and $V=(2I-B)^{-1}$. 

 \end{theorem} 
 
 Actually, according to the ideas in \cite{dario_beatrice}, the  scheme  \eqref{eq:SDA} allows to refine a given initial approximation to $E_B$, that is, if $E_B=\tilde{E}_{B}+D$, where $\tilde{E}_B$ is given and it satisfies $\|T(b)+\tilde{E}_B\|_{\infty}<1$, then SDA can be used to compute $D$. Indeed, if 
 $H$ in equation \eqref{eq:quadratic3}is replaced by $\tilde{E}_B+D$, it yields
 \begin{equation}\label{eq:eq:quadratic4}
 	D^2+(T(b)+\tilde{E}_B-2I)D+D(T(b)+\tilde{E}_B)+\tilde{R}=0,
 \end{equation}
 where $\tilde{R}=(T(b)+\tilde{E}_B)^2-2(T(b)+\tilde{E}_B)+A_1$. Analogously to the analysis above, we obtain the matrix pencil $\widehat{\M}-\lambda\widehat{N}$ such that 

\begin{equation*}
\widehat{\M}=\left[\begin{array}{cc}\tilde{S}A_1& 0\\-\tilde{S}\tilde{R} & I\end{array}\right],\quad  \widehat{\N}=\left[\begin{array}{cc}I &-\tilde{S}\\ 0 &\tilde{S}\end{array} \right],
\end{equation*}
where $\tilde{S}=(2I-T(b)-\tilde{E}_B)^{-1}$, and it holds 
\begin{equation*}%\label{eq:D}
\widehat{\M}\left[\begin{array}{c}I \\D  \end{array}\right]=\widehat{\N} \left[\begin{array}{c}I \\D  \end{array}\right]B, \quad \widehat{\M} \left[\begin{array}{c}\tilde{Y} \\I  \end{array}\right]\tilde{Z}=\widehat{\N} \left[\begin{array}{c}\tilde{Y} \\I \end{array}\right],
\end{equation*}   
where $\tilde{Y}=V(I-(T(b)+\tilde{E}_B)V)^{-1}, \tilde{Z}=(I-(T(b)+\tilde{E}_B)V)V(I-(T(b)+\tilde{E}_B)V)^{-1}$. 

Now  apply SDA  to the pencil $\widehat{M}-\lambda\widehat{N}$, we obtain the sequences defined as
\begin{equation}\label{eq:SDA2} 
\begin{aligned}
\tilde{E}_{k+1}&=\tilde{E}_k(I-\tilde{Q}_k\tilde{P}_k)^{-1}\tilde{E}_k\\
\tilde{P}_{k+1}&=\tilde{P}_k+\tilde{F}_k(I-\tilde{P}_k\tilde{Q}_k)^{-1}\tilde{P}_k\tilde{E}_k;\\
\tilde{F}_{k+1}&=\tilde{F}_k(I-\tilde{P}_k\tilde{Q}_k)^{-1}\tilde{F}_k;\\
\tilde{Q}_{k+1}&=\tilde{Q}_k+\tilde{E}_k(I-\tilde{Q}_k\tilde{P}_k)^{-1}\tilde{Q}_k\tilde{F}_k,
\end{aligned}
\end{equation}
where $P_0=E_0=\tilde{S}A_1$ and $Q_0=F_0=\tilde{S}$. 

Observe that $\rho(\tilde{Z})=\rho(V)<1$, then according to Lemma \ref{lem:conver_sda} it holds that $\lim_k\|\tilde{P}_k-D\|_{\infty}^{1/{2^k}}<\rho(B)\rho(V)<1$, that is, the sequence $\{\tilde{P}_k\}$ converges to $D$, so that $E_B=\tilde{E}_B+D$ is computed.  

\smallskip
One alternative is to set $\tilde{E}_B=(b(1) -T(b)\one)e_1^T$, where $\one=(1,1,\ldots)^T$ and $e_1=(1,0,\ldots)^T$, then  $T(b)+\tilde{E}_B$ is a nonnegative substochastic matrix such that $(T(b)+\tilde{E}_B)\one =b(1)\one$. Numerical experiments in Section \ref{ne} shows that there are cases where a reduction in  CPU time occurs when setting   $\tilde{E}_B=(b(1) -T(b)\one)e_1^T$ and applying iteration \eqref{eq:SDA2} for computing  $D$.
\smallskip

We mention that when applying the fixed-point iteration and  SDA to compute the correction part of a quasi-Toeplitz $M$-matrix, the computations rely on the package CQT-Toolbox of \cite{BSR} which implements the operations of semi-infinite quasi-Toeplitz matrices. In next section, we show that the the fixed-point iteration and SDA can be applied to a finite dimensional nonlinear matrix equation,  whose solution after extending to infinity is a good approximation to $E_B$. 
\section{Truncation to a finite dimensional matrix equation}

Recall that the correction part of a quasi-Toeplitz matrix $A=T(a)+E\in \QT_{\infty}$ satisfies $\lim_i\sum_{j=1}^{\infty}|e_{i,j}|=0$ for $E=(e_{i,j})_{i,j\in\mathbb Z^+}$.  Denote by $E^{(k)}$ the infinite matrix that coincides with the leading principal $k\times k$ submatrix of $E$ and is zero elsewhere, it follows form \cite[Lemma 2.9]{BMML2020} that there is a matrix $E^{(k)}$ such that $\lim_{k\rightarrow \infty}\|E-E^{(k)}\|_{\infty}=0$.  

 For an invertible $M$-matrix $A=I-A_1\in \QT_{\infty}$, suppose  $(I-T(b)-E_B)^2=A$, then for $E_B$ and a given $\epsilon>0$,  there is a  sufficiently large $k$  such that 
 \begin{equation}\label{eq:bound}
 \|E_B^{(k)}-E_B\|_{\infty}<\epsilon.
 \end{equation} 
If we partition $E_B$ into $E_B=\left(\begin{array}{cc}E_{11} & E_{12}\\ E_{21}& E_{22} \end{array}\right)$, where $E_{11}$ is  the principal $k\times k$ submatrix of $E_B$, $E_{12}\in \mathbb R^{k\times \infty}$, $E_{21}\in \mathbb R^{\infty\times k}$ and $E_{22}\in \mathbb R^{\infty\times \infty}$, it follows from $\|E_B^{(k)}-E_B\|_{\infty}<\epsilon$ that $\|E_{12}\|_{\infty}<\epsilon, \|E_{21}\|_{\infty}<\epsilon$ and $\|E_{22}\|_{\infty}<\epsilon$. 

Let $W=2T(b)-A_1-T(b)^2$, then $T(b)$ and $W$ can be partitioned into $T(b)=\left(\begin{array}{cc}T_{11}& T_{12}\\ T_{21}& T_{22} \end{array}\right)$ and $W=\left(\begin{array}{cc}W_{11}& W_{12}\\ W_{21}& W_{22} \end{array}\right)$, where $T_{11}$ and $W_{11}$ are, respectively, the principal $k\times k$ submatrices of $T(b)$ and $W$. Substituting $E_B, T(b)$ and $W$ into the equation $(I-T(b)-E_B)^2=I-A_1$, we get
\begin{equation}\label{eq:finite1}
E_{11}^2-(I_k-T_{11})E_{11}-E_{11}(I_k-T_{11})=W_{11}-E_{12}E_{21}-E_{12}T_{21}-T_{12}E_{21},
\end{equation}
where $I_k$ is the  identity matrix of size $k$. 

Consider the matrix equation 

\begin{equation}\label{eq:finite2}
G^2-(I_k-T_{11})G-G(I_k-T_{11})=W_{11},
\end{equation}
which is equivalent to 
\begin{equation}\label{eq:finite22}
(I_k-T_{11}-G)^2=I-A_{11}-T_{12}T_{21},
\end{equation}
where $A_{11}$ is the principal $k\times k$ submatrix of $A_1$.  Observe that $A_{11}\geq 0$ and $T_{12}T_{21}\geq 0$, if in addition  $\rho(A_{11}+T_{12}T_{21})<1$, which can be verified if $\|A_{11}+T_{12}T_{21}\|_{\infty}<1$,  then  $I-A_{11}-T_{12}T_{21}$ is  a nonsingular $M$-matrix.   In what follows we assume  $\|A_{11}+T_{12}T_{21}\|_{\infty}<1$, then $I-A_{11}-T_{12}T_{21}$  admits a unique $M$-matrix as a square root (see \cite[Theorem 6.18]{higham}), so that equation \eqref{eq:finite22}, as well as equation \eqref{eq:finite2}, has a unique solution $G$ such that $T_{11}+G\geq 0$ and $\rho(T_{11}+G)<1$. In fact, analogously to \cite[Theorem 3.1]{qtm_laa}, it is can be seen that $\|T_{11}+G\|_{\infty}<1$. 

\smallskip

Subtracting equation \eqref{eq:finite1} form  equation \eqref{eq:finite2} yields
\begin{equation}\label{eq:delta1}
G^2-E_{11}^2-(G-E_{11})(I_k-T_{11})-(I_k-T_{11})(G-E_{11})=\Delta W,
\end{equation}
where  $\Delta W=E_{12}E_{21}+E_{12}T_{21}+T_{12}E_{21}$.  It can be seen that 
\begin{align}\label{eq:W}
\|\Delta W\|_{\infty}&=\|E_{12}E_{21}+E_{12}T_{21}+T_{12}E_{21}\|_{\infty}\notag \\
&\leq \epsilon^2+\|T_{21}\|_{\infty}\epsilon+\|T_{12}\|_{\infty}\epsilon\notag\\
&\leq (2\|b\|_{\w}+\epsilon)\epsilon,
\end{align}
where the last inequality holds as $\|T_{12}\|_{\infty}\leq \|T(b)\|_{\infty}=\|b\|_{\w}$ and  $\|T_{21}\|_{\infty}\leq \|T(b)\|_{\infty}=\|b\|_{\w}$.
\smallskip

On the other hand, a direct computation of equation \eqref{eq:delta1} yields
\begin{equation}\label{eq:delta2}
(2I_k-T_{11}-G)(G-E_{11})-(G-E_{11})(T_{11}+E_{11})=-\Delta W. 
\end{equation}
 Observe that $2I_k-T_{11}-G$ is a nonsingular $M$-matrix as $T_{11}+G\geq 0$ and $\rho(T_{11}+G)<1$. Moreover, we have $\|T_{11}+E_{11}\|_{\infty}<1$  as $T_{11}+E_{11}$ is the principal $k\times k$ submatrix of $T(b)+E_B$ and $\|T(b)+E_B\|_{\infty}<1$. Then one can check that
  \begin{equation}\label{eq:delta3}
 G-E_{11}=-\sum_{j=1}^{\infty}(2I_k-T_{11}-G)^{-j-1}\Delta W (T_{11}+E_{11})^j
 \end{equation}
 is well defined and it solves equation \eqref{eq:delta2}. 
 \smallskip

 Let $\alpha=\|(2I_k-T_{11}-G)^{-1}\|_{\infty}$ and $\beta=\|T_{11}+E_{11}\|_{\infty}$, we have $\alpha=\frac{1}{2}\|\sum_{j=0}^{\infty}(\frac12(T_{11}+G))^j\|_{\infty}\leq \frac{1}{2-\|T_{11}+G\|_{\infty}}$, so that $\alpha\beta\leq \frac{\beta}{2-\|T_{11}+G\|_{\infty}}<1$ since $\|T_{11}+G\|_{\infty}< 1$ and $\beta<1$. 
 Then we deduce from \eqref{eq:W} and \eqref{eq:delta3} that 
\begin{align}\label{ine:bound1}
\|G-E_{11}\|_{\infty}&\leq \sum_{j=1}^{\infty}(\alpha\beta)^j\alpha \|\Delta W\|_{\infty}\notag \\
%&\leq \frac{\alpha\beta}{1-\alpha\beta}\|\Delta W\|_{\infty} \notag\\
&\leq \frac{\alpha}{1-\alpha\beta}(2\|b\|_{w}+\epsilon)\epsilon. 
\end{align}

Let $E_G$  be the matrix that coincides in the leading principal $k\times k$ submatrix with $G$ and is zero elsewhere, then we have from \eqref{eq:bound} and \eqref{ine:bound1} that 

\begin{align}\label{ine:bound}
\|E_G-E_B\|_{\infty}&\leq \|E_G-E_B^{(k)}\|_{\infty}+\|E_B^{(k)}-E_B\|_{\infty}\notag\\
&\leq \|G-E_{11}\|_{\infty}+\epsilon\notag \\
&\leq(1+ \frac{\alpha}{1-\alpha\beta}(2\|b\|_{w}+\epsilon))\epsilon.
\end{align}

Hence, we can see from \eqref{ine:bound} that for a given $\epsilon>0$ and sufficiently large $k$, if $\alpha\beta\leq c<1$ for some constant $c$, then $E_G$ may serve as a good approximation to $E_B$. This  implies  that the correction part $E_B$ can be approximated by firstly computing the numerical solution of  equation \eqref{eq:finite2} and then extending the computed solution to infinity.  

 It is not difficult to see  that  the fixed-point iteration \eqref{ifpi} and SDA can be applied  to equation \eqref{eq:finite2} for computing the solution $G$. Numerical experiments in next section show that when the size $k$ is small, it is efficient to  approximate the correction part $E_B$ by computing the solution of equation \eqref{eq:finite2} and extending it to infinity, while when $k$ is large, that is, the coefficients are large-scale matrices, both fixed-point iteration and SDA lose the effectiveness.  
\smallskip

We provide some insight on how to select integer  $k$ such that the matrix  $G$ of size $k\times k$, after extending to infinity, is approximate enough to $E_B$.  Observe that the substitution of $E_G$ into the equation $(I-T(b)-X)^2=A$ yields
\begin{align*}
A-(I-T(b)-E_G)^2=\left(\begin{array}{cc} 0& GT_{12}-W_{12}\\ T_{21}G_{11}-W_{21} & -W_{22}, \end{array}\right),
\end{align*}
from which we see that $E_G$ is a good approximation to $E_B$ if   $\|GT_{12}-W_{12}\|_{\infty}<c\epsilon$, $\|T_{21}G-W_{21}\|_{\infty}<c\epsilon$ and $\|W_{22}\|_{\infty}<c\epsilon$ for some constant $c$ and a given $\epsilon>0$. It can be seen that these inequalities hold if  
\begin{equation}\label{ine:1}
\|GT_{12}\|<c_1\epsilon, 
\end{equation}
\begin{equation}\label{ine:2}
\|T_{21}G\|_{\infty}<c_2\epsilon,
\end{equation}
and
\begin{equation}\label{ine:3}
 \max\{\|W_{12}\|, \|W_{21}\|, \|W_{22}\|_{\infty}\}<c_3\epsilon,
 \end{equation}
for some constants $c_1, c_2$ and $c_3$. Hence, we can choose $k$  such that inequalities \eqref{ine:1}-\eqref{ine:3} are satisfied. 

Actually, since  $W$ is a correction matrix, one can check that  inequality \eqref{ine:3} holds if we choose $k$   such that $\|W-W^{(k)}\|_{\infty}<\epsilon$, where $W^{(k)}$  is the infinite matrix that coincides with the leading principal $k\times k$ submatrix of $W$ and is zero elsewhere. Hence, if the matrix $W$ has a nonzero part of size   $n_1\times n_2$, we can choose $k$ such that $k>\max\{n_1,n_2\}$.

We next show how to choose $k$ such that inequalities \eqref{ine:1} and \eqref{ine:2} hold.  Observe that for  $\epsilon>0$, there is $N\in \mathbb Z^+$ such that $\|E_B-E_B^{(n)}\|_{\infty}<\epsilon$ for any $n\geq N$. Set $k>N$ and $G=\left(\begin{array}{cc} G_{11} &G_{12}\\ G_{21}& G_{22} \end{array}\right)\in \mathbb R^{k\times k}$, where $G_{11}\in \mathbb R^{N\times N}, G_{12}\in \mathbb R^{N\times (k-N)}, G_{21}\in \mathbb R^{(k-N)\times N}$ and $G_{22}\in \mathbb R^{(k-N)\times (k-N)}$.  Observe that 
$$\|E_G-E_B^{(N)}\|_{\infty}\leq \|E_G-E_B\|_{\infty}+\|E_B-E_B^{(N)}\|_{\infty},$$
which, together with inequality \eqref{ine:bound} and the fact $\|E_B-E_B^{(N)}\|_{\infty}<\epsilon$, implies  that $\|E_G-E_B^{(N)}\|_{\infty}<\tilde{c}_1\epsilon$ for some constant $\tilde{c}_1$. On the other hand, observe that  $E_G-E_B^{(N)}$ coincides in the leading principal $k\times k$ submatrix  with $\left(\begin{array}{cc}*& G_{12}\\ G_{21}&G_{22}\end{array}\right)$ and is zero elsewhere, where $*$ is an $N\times N$ matrix,  we thus have $\|G_{12}\|_{\infty}<\tilde{c}_1\epsilon$, $\|G_{21}\|_{\infty}<\tilde{c}_1\epsilon$ and $\|G_{22}\|_{\infty}<\tilde{c}_1\epsilon$.

 \smallskip

Suppose $b(z)=\sum_{j=-q}^pb_jz^j$, then from the partition of $T(b)$ we know that $T_{12}=\left(\begin{array}{cc} O\\ \tilde{T}\end{array}\right)$, where $O$ is a zero matrix of size $(k-p)\times \infty$ and $\tilde{T}$ is a $p\times \infty$ matrix with a $p\times p$ nonzero submatrix located in the  bottom leftmost corner. If $k$ is selected such that $k-N>p$, we have from $\|G_{12}\|_{\infty}<\tilde{c}_1\epsilon$ and $\|G_{21}\|_{\infty}<\tilde{c}_1\epsilon$ that $\|GT_{12}\|_{\infty}\leq \max\{\|G_{12}\|_{\infty},\|G_{21}\|_{\infty}\}\|\tilde{T}\|_{\infty}<c_{1}\epsilon$ for some constant $c_1$. Similarly, if $k-N>q$, inequality \eqref{ine:2} holds.

The above analysis indicates that if the matrix $W$ has a nonzero part of size   $n_1\times n_2$ and the symbol $b$ of $T(b)$ is a Laurent series $b(z)=\sum_{j=-q}^pb_jz^j$, then we can choose $k$ such that
 \begin{equation}\label{N}
 k-N>p, K-N>q\  {\rm and}\ k>\max\{n_1,n_2\}. 
 \end{equation}
Observe that the value of $N$ in \eqref{N} is unknown, hence  we can obtain a necessary condition for determining  $k$, that is, $k>\max\{p,q,n_1,n_2\}$. In  our  numerical experiments, we have set $k=3\max\{p,q,n_1,n_2\}$ and it seems sufficient.

Note that equation \eqref{eq:finite1} is a special case of the following equation 
\[
X^2-AX-XA=B,
\] 
where $A$ is a large-scale nonsingular $M$-matrix with an almost Toeplitz structure, and $B$ is a low-rank matrix. It seems interesting to investigate whether there are more efficient algorithms for computing the  solution by exploiting the quasi-Toeplitz structure of $A$ and the low-rank structure of matrix $B$. We leave this as a future consideration.
 
 \section{Numerical experiments}\label{ne}
In this section, we show by numerical experiments the effectiveness of the fixed-point iteration \eqref{ifpi} and SDA. The computations of  semi-infinite quasi-Toeplitz matrices rely on the package CQT-Toolbox \cite{BSR}, which can be downloaded at https://github.com/numpi/cqt-toolbox, while computation of   the solution of equation \eqref{eq:finite2} is implemented relying on the standard finite size matrix operations. The tests were performed in MATLAB/version R2019b on the Dell Precision 5570 with an Intel Core i9-12900H and 64 GB main memory. We set the internal precision in the computations to {\tt threshold = 1.e-15}. For each experiment, the iteration is terminated if $\|(I-T(b)-X)^2-A\|_{\infty}/\|A\|_{\infty}\leq {\tt 1.e-13}$. The code is available from the authors upon request.

We recall that a quasi-Toeplitz matrix $A=T(a)+E_A$ is representable in MATLAB relying on the CQT-toolbox \cite{BSR} by {\tt A=cqt(an,ap,E)}, where  the vectors {\tt an} and {\tt ap} contain the coefficients of the symbol $a(z)$
with non negative and non positive indices, respectively, and $E$ is a finite matrix representing the  non zero part  of the correction $E_A$.

\begin{example}\label{exm:E}
 Let $A=I-S$ with $S=\tilde{S}/(\|\tilde{S}\|_{\infty}+1)$, where the construction of  $\tilde{S}$ in MATLAB is done as 
 ${\tt \tilde{S}=cqt(s_n,s_p, E_{\tilde{S}})}$.  We set ${\tt s_n=rand(32,1)}$, ${\tt s_p=rand(30,1)}$, ${\tt s_n(1)=s_p(1)=1}$.  For the frist test, we set 
 ${\tt E_{\tilde{S}}= rand(0,0)}$, while for the second test, we set ${\tt E_{\tilde{S}}=rand(1000,1000)}$. 
 \end{example}

Suppose $B=T(b)+E_B$ is such that $(I-B)^2=A$, we first compute by Algorithm \ref{alg:b} an approximation $\hat{b}(z)=\sum_{j=-n+1}^n\hat{b}_jz^j$ to the symbol $b(z)$ of $T(b)$,  then we apply the fixed-point iteration \eqref{ifpi} and SDA to compute $E_B$. In Figure \ref{fig:symbol}  we show  the graph of the computed coefficients $\hat{b}_j$, $j=-n+1,\ldots,n$. In Figure \ref{fig:correction}, we show the correction part $E_B=(e_{i,j})_{i,j\in \mathbb Z^+}$ in logarithmic scale, which is obtained by the fixed-point iteration. The number of iterations,  CPU times required in the computations and the relative residuals  are reported in Table \ref{tab_res}. In Table \ref{tab:band} we report the features of the computed $E_B$ computed by the fixed-point iteration, including band of the Toeplitz part,  the rank and  the number of the nonzero rows and columns of the correction part. 

 It can be seen from Table \ref{tab_res} that  the number of iterations required by SDA is much less than the number of iterations required by the fixed-point iteration. Concerning the CPU time,  we can see that   the fixed-point iteration  takes less time than SDA in Test 1, while in Test 2, the CPU time taken by SDA is about 1/3 of that taken by the fixed-point iteration. 
Moreover, in test 1, when  applying  SDA  to compute matrix $D$ such that $E_B=D+\tilde{E}_B$, where $\tilde{E}_B=(s(1) -T(s)\one)e_1^T$, it takes 119.56s, which provides a reduction in CPU time comparing  with the case where SDA is applied directly for the computation of $E_B$.

%\pgfdeclareimage[width=6cm]{gm}{gm}
%\pgfdeclareimage[width=6cm]{gp}{gp}

\begin{figure}\center
\includegraphics[width=74mm]{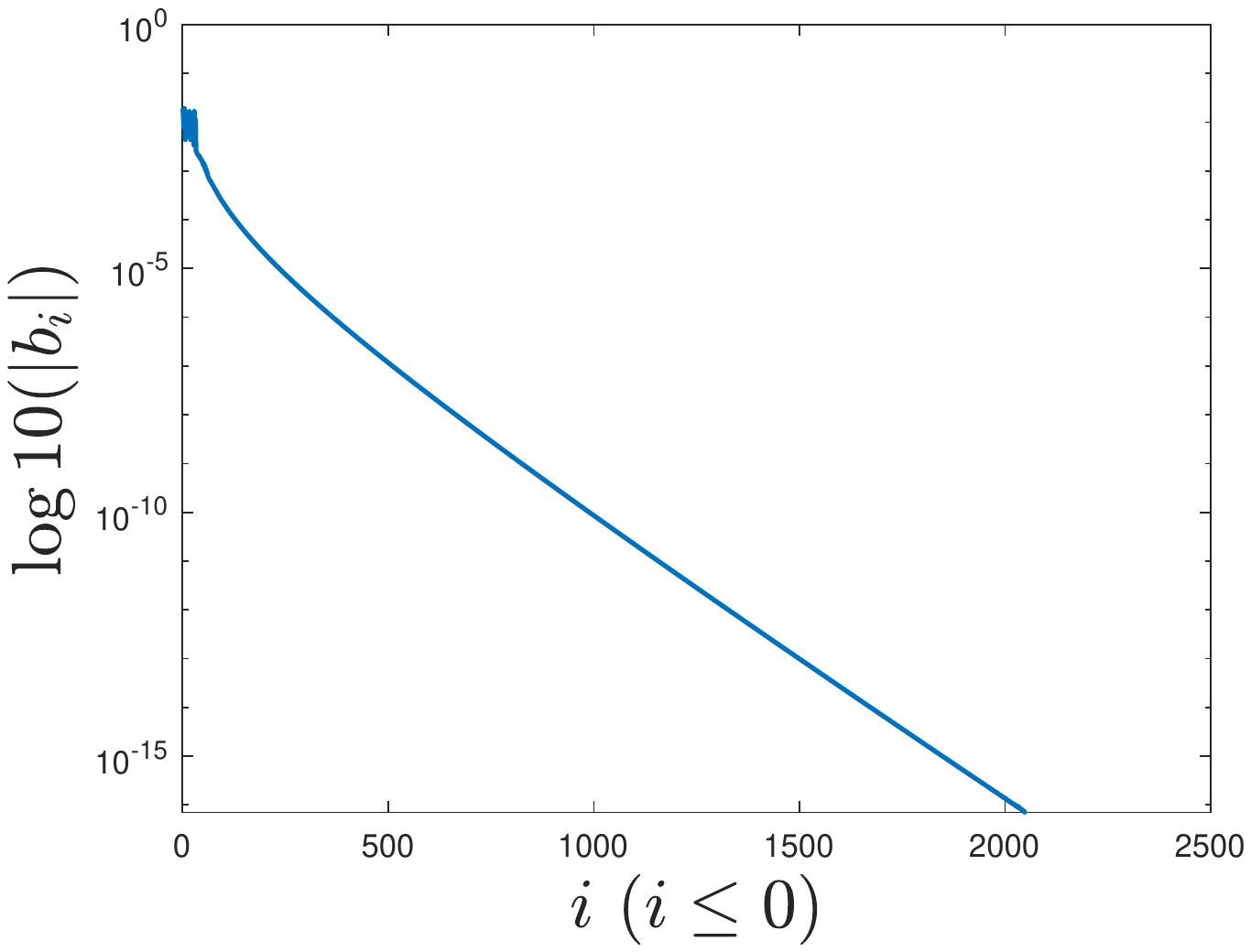}
\includegraphics[width=74mm]{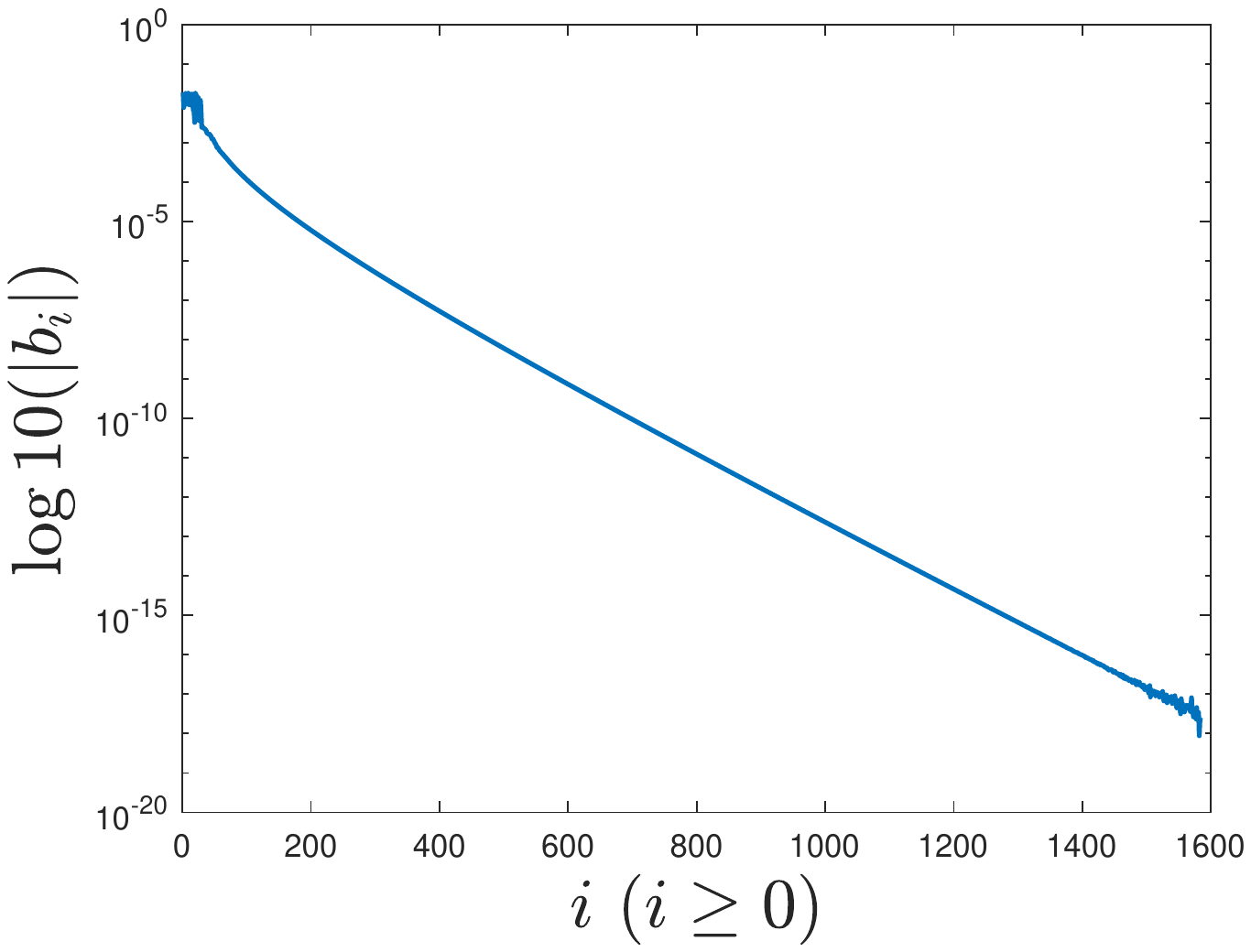}

\caption{Toeplitz part  of the computed  $B=T(b)+E_B$ in Test 1: the log-scale   of the absolute value of  coefficients $b_i$ of the symbol $b(z)$ for $i\leq 0$ (left) and for $i\geq 0$ (right). The coefficients are computed by Algorithm \ref{alg:b}.}\label{fig:symbol}

\end{figure}

%\pgfdeclareimage[width=8cm]{correction}{correction}
\begin{figure}
\center
\includegraphics[width=80mm]{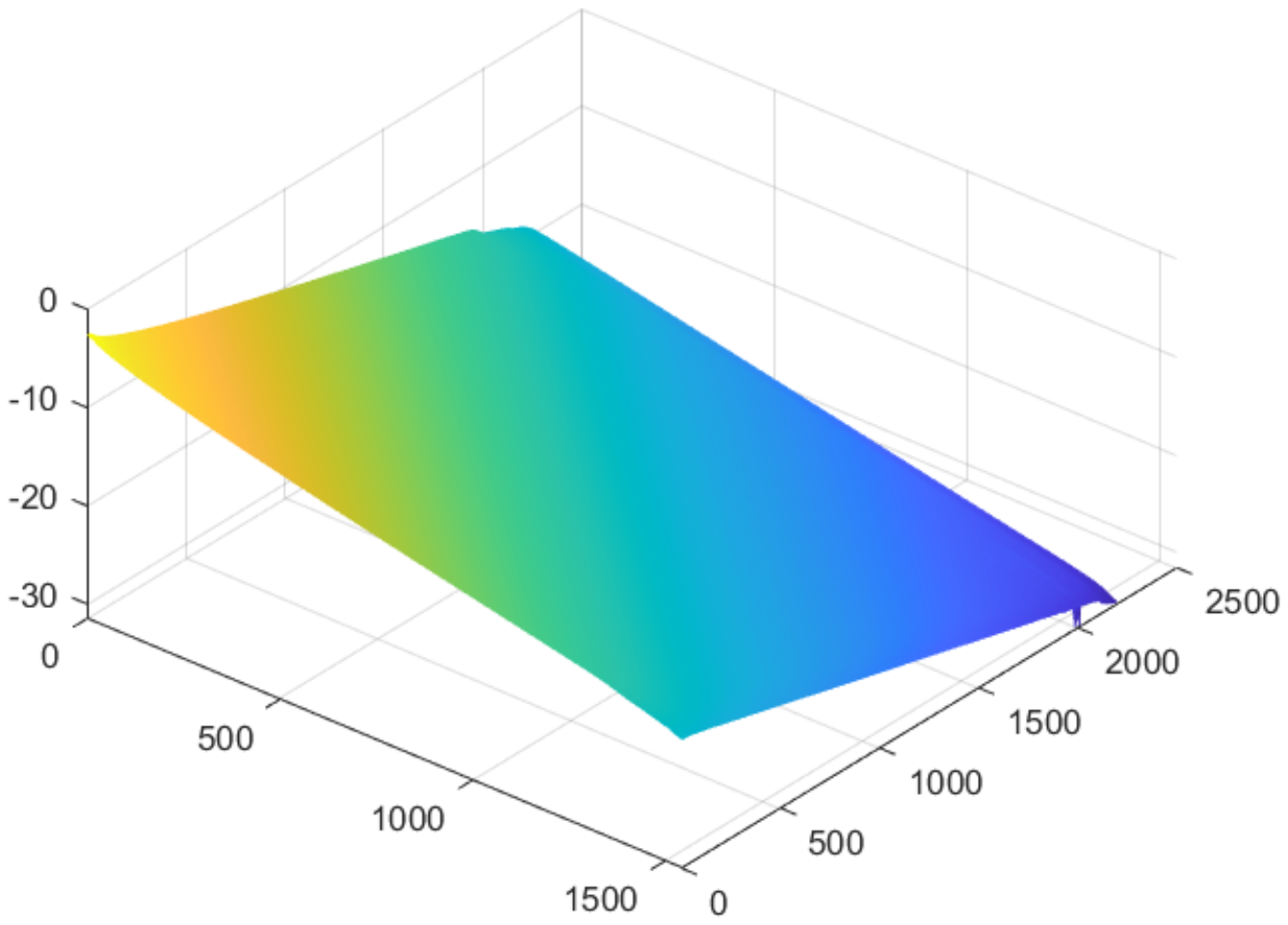}
\caption{The correction part $E_B$ in Test 1: absolute value of $E_B$ in log scale, where $E_B$ is computed by the fixed-point iteration \eqref{ifpi}. } \label{fig:correction}
\end{figure}

 \begin{table}
\begin{center}
\begin{tabular}{c|ccc|ccc}
\multicolumn{1}{c}{}&\multicolumn{3}{c}{Test 1} &\multicolumn{3}{c}{Test 2}\\ \hline
Iterations&res.&iter.&time& res.  &iter.&time \\\hline
FPI& {\tt 7.02e-14}& 55&{\tt 114.77 }&  {9.62e-14}&54 & {\tt 252.14}\\
SDA&  {\tt 4.42e-14}& 6& {\tt 170.07}& {\tt 6.61e-14}  &6& {\tt 81.01} \\
%FPI+NI& {\tt 6.82e-14}& 39& {\tt 338.41}& {\tt 2.06e-14}&  42 &{\tt 347.99} 
\end{tabular}
\caption{Relative residual, number of iterations, CPU time in seconds in the computation $E_B$. FPI means the fixed-point iteration.}\label{tab_res}
\end{center}
\end{table}

 \begin{table}
\begin{center}
\begin{tabular}{c|c|c}
\multicolumn{1}{c}{}&\multicolumn{1}{c}{Test 1} &\multicolumn{1}{c}{Test 2}\\\hline
Band&4200 & 376\\
Rows& 2799 & 1296\\
Columns&1319 &1162 \\
Rank& 80& 1026
\end{tabular}
\caption{Features of matrix $B=T(b)+E_B$ in Test 1 which is computed by FPI, including the band of the Toeplitz part $T(b)$, number of nonzero rows and columns, and  rank of the correction of the computed $E_B$.} \label{tab:band}
\end{center}
\end{table}

\begin{example}
Let $A=I-S$ with $S=T(s)+E_{S}\in \QT_{\infty}$, where $T(s)=s_0I$ with $s_0<1$ and $E_{B}$ is the correction matrix with a $(p+m+n)\times (p+m+n)$ leading submatrix $E_S^{P}$ and zero elsewhere. Here, $E_S^P=\left(\begin{array}{ccc}V_p& & \\ & O_m& \\ & & -s_0I_n\end{array}\right)$, where $O_m$ is the zero matrix of size $m\times m$, $I_n$ is the identity matrix of size $n$, and the matrix $V_p=\left(\begin{array}{cc} U_{p\times q}\\ O_{(q-p)\times q}\end{array}\right)$ is a $q\times q$ block matrix with
\begin{equation*}
U_{p\times q}=\left(\begin{array}{cccccc}
u_{11} & u_{12}& \cdots &  u_{1p}& \cdots & u_{1q}\\
0& u_{22} & \cdots & u_{2p}&\cdots & u_{2q}\\
\vdots & \ddots &\ddots &\vdots& \ddots& \vdots \\
0& 0& \cdots & u_{pp}&\cdots  & u_{pq}
\end{array}
\right)_{p\times q}.
\end{equation*}
where $u_{ii}=-s_0$ for $i=1,\ldots, p$, and $u_{i,j}\geq 0$ for $i=1,2,\ldots,p$ and $j=i+1\ldots,q$. Moreover, for $i=1,2,\ldots, p$, it satisfies that $\sum_{j=i+1}^qu_{ij}<1$. 

\end{example}

 \begin{table}
\begin{center}
\begin{tabular}{c|c|c|c|c| c}
%\multicolumn{1}{c}{}&\multicolumn{1}{c}{$\delta=1$} &\multicolumn{1}{c}{$\delta=0.5$}\\\hline
Test& $s_0$& $m$ & $n$& $p$& $q$\\
\hline
1& 0.1 & 100 &1000&  1 & 100\\
2& 0.5& 100 & 1500 & 2& 100\\
3&  0.9& 100& 2000& 2& 100
\end{tabular}
\caption{Different values of the parameters $s_0, m, n, p$ and $q$.}\label{tab:tests}
\end{center}
\end{table}

For different values of the parameters $s_0, m, n, p$ and $q$ as listed in Table \ref{tab:tests}, we apply the  fixed-point iteration \eqref{ifpi} and SDA to compute the matrix $E_B$ such that  $(I-T(b)-E_B)^2=A$. It can be seen that the symbol $b(z)$ satisfies $(1-b(z))^2=1-s_0$,  which, together with the fact that $\|b\|_{\w}=\|T(b)\|_{\infty}<1$, implies  $b(z)=1-\sqrt{1-s_0}$, so that $T(b)$ is a diagonal matrix with diagonal elements being $1-\sqrt{1-s_0}$.

In this example, we observe that $E_B$ can be obtained by the fixed-point iteration  as well as  SDA  in just one or two steps. We also implement the Binomial iteration (BI) and the CR in \cite{qtm_laa} for computing the the whole matrix $B=T(b)+E_B$,  the CPU time and residual error are compared with the fixed-point iteration and SDA in the computation of $E_B$, and are reported in Table \ref{tab:comparingwithCR}. We mention that  the residual error for BI and CR is obtained by $r=\|(I-\hat{Y})^2-A\|_{\infty}/\|A\|_{\infty}$, where $\hat{Y}$ is the computed square root.

 \begin{table}
\begin{center}
\begin{tabular}{c|cc|cc|cc}
\multicolumn{1}{c}{}&\multicolumn{2}{c}{Test 1} &\multicolumn{2}{c}{Test 2}& \multicolumn{2}{c}{Test 3}\\ \hline  
Algorithms &time &res & time &res &time  &res \\\hline
FPI& 2.7734 & 1.01e-15 &  19.44&3.00e-15  &34.28 &3.41e-14 \\
SDA&  8.4440 &  1.40e-15  & 25.25 &2.35e-15  &  44.80 & 6.79e-14  \\
CR & 11.2681 &1.83e-15 & 44.96 &4.59e-15 & 95.91&7.48e-14\\
BI& 14.3910 & 7.65e-16 &48.40 &2.02e-15 &111.53&5.22e-14
\end{tabular}
\caption{Comparison of the fixed-point iteration \eqref{ifpi} and SDA in computing $E_B$ with the Binomial iteration and CR algorithm in computing $B$: the CPU time in seconds and relative residual in the computations.}\label{tab:comparingwithCR}
\end{center}
\end{table}

As we can see from Table \ref{tab:comparingwithCR}, the  fixed-point iteration \eqref{ifpi} and SDA   take less CPU time comparing with the Binomial iteration and CR algorithm. Moreover, the fixed-point iteration \eqref{ifpi}, comparing with CR algorithm,  has a speed-up in the CPU time by a factor of about 4 in Test 1 and 2.5 in Tests  2 and 3.

\begin{example}\label{exm:E}
Let  $A=cI-T(s)$ with $T(s)={\tt cqt(s_n,s_p)}$,  where $c$, ${\tt s_n}$ and ${\tt s_p}$  are constructed in MATLAB as 

${\tt s_p= rand(p,1)}$, ${\tt s_n=rand(q,1)}$, ${\tt s_n(1)=s_p(1)=1}$, ${\tt c=sum(s_n)+sum(s_p)}$.
\end{example}

It can be seen that $\|T(s)\|_{\infty}=\|s\|_{\w}<c$, so that $A$ is an invertible $M$-matrix. 
For different values of $p$ and $q$, we  apply the fixed-point iteration \eqref{ifpi} and SDA for computing  matrix $E_B$ such that $c(I-T(b)-E_B)^2=A$, where the symbol $b(z)$ is approximated by $\hat{b}(z)$ that is computed by Algorithm \ref{alg:b}. %In view of property (i) of Lemma \ref{summary}, we know that the matrix $E_B$ satisfies
%\[
%\begin{aligned}
%E_B^2-(I-T(b))E_B-E_B(I-T(b))-H_1H_2=0,
%\end{aligned}
%\]
%where $(H_1)_{i,j}=(b_{-i-j+1})_{i,j\in \mathbb Z^+}$ and $(H_2)_{i,j}=(b_{i+j-1})_{i,j\in \mathbb Z^+}$.  

%Suppose $\hat{b}(z)=\sum_{j=-n+1}^n\hat{b}_jz^j$ is the Laurent series such that $\|b-\hat{b}\|_{\w}<\epsilon$ for the given $\epsilon>0$, 

%In the finite case, as we can see, the product  of two Toeplitz matrices is a Toeplitz matrix plus two correction matrix, one of which is located in the upper-left conner of the product and the other is located in the lower-right conner. Observe that only the one in the top leftmost corner is useful when extending the computed square root to the infinity. 
%In this sense, for a given quasi-Toeplitz $M$-matrix, suppose the computed square root is $B=T(b)+E_B$ with $b(z)=\sum_{i=-q}^pb_iz^i$ and $E_B$ has a finite support of size $n\times r$, when truncating $A$ to a finite size, say $m\times m$ matrix, motivated by the goal to keep well separated the compact correction in the top leftmost corner from that in the bottom rightmost corner of the finite size matrix, we may choose $m=3\times \max\{r,n,p,q\}$ as did in the reference paper \cite{arXiv}.

We also apply the fixed-point iteration and SDA to equation \eqref{eq:finite2} for computing its solution $G$, so that $E_B$ can be approximated by extending $G$ to infinity. Table \ref{compar} reports the CPU time taken by the fixed-point iteration and SDA  when applied to matrix equation \eqref{eq:finite2}, as well as the CPU time needed in the computation of the $E_B$ relying on the operations of quasi-Toeplitz matrices. 

We  observe from Table \ref{compar} that when the values of $p$ and $q$ are both small, say $p=4, q=2$, it seems that applying the fixed-point iteration \eqref{ifpi} and  SDA  to the truncated matrix equation \eqref{eq:finite2} takes less CPU time. For different values of $p$ and $q$ listed in Table \ref{compar}, the rank of the correction matrix is $k$=501, 1539, 8496 and 3834, respectively, we observe that when $k$ becomes large, the algorithms applied to the truncated matrix equation \eqref{eq:finite2} take more CPU times,  and it can be seen that the algorithms relying on operations of quasi-Toeplitz matrices are more efficient. 
\begin{table}
\begin{center}
\begin{tabular}{c|c|c}
%\multicolumn{1}{c}{}&\multicolumn{2}{c}{Test 1} &\multicolumn{2}{c}{Test 2}& \multicolumn{2}{c}{Test 3}\\ \hline  
   ($p,q$) &FPI& SDA\\\hline
 (4,2)&$2.79\cdot 10^{-2}$ $[1.09\cdot 10^{-1}] $& $1.95\cdot 10^{-2}$ $[1.13\cdot 10^{-1}]$ \\
(12,10) & $4.19\cdot 10^0$ $[3.48\cdot 10^0]$ &  $2.45\cdot 10^0$ $[5.61\cdot 10^0]$\\
 (20,2) &$6.90\cdot 10^2$ $[4.63\cdot 10^1]$ &  $6.56\cdot 10^2$ $[7.17\cdot 10^1]$\\
(20,20)&$7.52\cdot 10^1$ $ [2.49\cdot 10^1]$ & $3.53\cdot 10^1$ $[4.23\cdot 10^1 ]$ 
\end{tabular}
\caption{CPU time in seconds, needed by the fixed-point iteration and SDA for  computing a $k\times k$ matrix, which, after extending to infinity, is a good approximation to  $E_B$. For comparison, the CPU time needed by FPI and SDA relying on the operations of quasi-Toeplitz matrices is written between bracket.}\label{compar}
\end{center}
\end{table}

\section{Conclusions}
We have fully exploited the quasi-Toeplitz structure in the computation of the square root of invertible quasi-Toeplitz $M$-matrices. 
 We  propose algorithms for computing the Toeplitz part and the correction part respectively. The Toeplitz part is computed  by  Algorithm \ref{alg:b} at the basis of evaluation/interpolation at the $2n$ roots of unique. We propose a fixed-point iteration and a structure-preserving doubling algorithm for the computation of the correction part. Moreover, we show that the correction part can be approximated by extending the solution of a nonlinear matrix equation to infinity.  Numerical experiments  show that SDA in general takes less CPU time than the fixed-point iteration. There are also cases where the fixed-point iteration is inferior to  SDA. There are cases where both the fixed-point iteration and SDA work better than the Binomial iteration and CR algorithm that exploit the quasi-Toeplitz structure indirectly.

\end{document}